\documentclass[12pt,a4paper,reqno]{amsart}
\usepackage{amssymb}
\usepackage{amscd}
\usepackage{psfig}
\numberwithin{equation}{section}

\theoremstyle{plain}

\newtheorem{theorem}[subsection]{Theorem}

\newtheorem{lemma}[subsection]{Lemma}

\theoremstyle{definition}

\newtheorem{remark}[subsection]{Remark}

\renewcommand{\leq}{\leqslant}
\renewcommand{\geq}{\geqslant}

\newsavebox{\proofbox}
\savebox{\proofbox}{\begin{picture}(7,7)%
  \put(0,0){\framebox(7,7){}}\end{picture}}

\newcommand\lin{{\operatorname{lin}}}

\newcommand\R{\mathbb{R}}

\newcommand\eps{\varepsilon}

\begin{document}

\title[High exponent limit for NLW]{The high exponent limit $p \to \infty$ for the one-dimensional nonlinear wave equation}
\author{Terence Tao}
\address{Department of Mathematics, UCLA, Los Angeles CA 90095-1555}
\email{tao@@math.ucla.edu}

\vspace{-0.3in}
\begin{abstract}
We investigate the behaviour of solutions $\phi = \phi^{(p)}$ to the one-dimensional nonlinear wave equation $-\phi_{tt} + \phi_{xx} = -|\phi|^{p-1} \phi$ with initial data $\phi(0,x) = \phi_0(x)$, $\phi_t(0,x) = \phi_1(x)$, in the high exponent limit $p \to \infty$ (holding $\phi_0, \phi_1$ fixed).  We show that if the initial data $\phi_0, \phi_1$ are smooth with $\phi_0$ taking values in $(-1,1)$ and obey a mild non-degeneracy condition, then $\phi$ converges locally uniformly to a piecewise limit $\phi^{(\infty)}$ taking values in the interval $[-1,1]$, which can in principle be computed explicitly.
\end{abstract}

\maketitle

\section{Introduction}

Consider solutions $\phi: \R \times \R \to \R$ to the defocusing nonlinear wave equation
\begin{equation}\label{nlw}
-\phi_{tt} + \phi_{xx} = |\phi|^{p-1} \phi
\end{equation}
where $p>1$ is a parameter.  From standard energy methods (see e.g. \cite{sogge:wave}), relying in particular on the conserved energy
\begin{equation}\label{energy}
E(\phi)(t) = \int_{\R} \frac{1}{2} |\phi_t|^2 + \frac{1}{2} |\phi_x|^2 + \frac{1}{p+1} |\phi|^{p+1}\ dx
\end{equation}
and on the Sobolev embedding $H^1_x(\R) \subset L^\infty_x(\R)$, we know that given any initial data $\phi_0 \in H^1_x(\R)$, $\phi_1 \in L^2_x(\R)$, there exists a unique global energy class solution $\phi \in C^0_t H^1_x \cap C^1_t L^2_x(\R \times \R)$ to \eqref{nlw} with initial data $\phi(0) = \phi_0$, $\phi_t(0) = \phi_1$.  One has a similar theory for data that is only locally of finite energy, thanks to finite speed of propagation.

In this paper we investigate the asymptotic behaviour of this solution $\phi = \phi^{(p)}$ in the high exponent limit\footnote{We are indebted to Tristan Roy for posing this question.} $p \to \infty$, while keeping the initial data fixed.  To avoid technicalities, let us suppose that $\phi_0, \phi_1$ are smooth and compactly supported, and that $|\phi_0(x)| < 1$ for all $x$.  Formally, we expect $\phi^{(p)}$ to converge in some sense to some solution $\phi = \phi^{(\infty)}$ of the ``infinitely nonlinear defocusing wave equation''
\begin{equation}\label{nlw-infty}
-\phi_{tt} + \phi_{xx} = |\phi|^{\infty} \phi
\end{equation}
with initial data $\phi(0) = \phi_0$, $\phi_t(0) = \phi_1$.

Of course, \eqref{nlw-infty} does not make rigorous sense.  But, motivated by analogy with infinite barrier potentials\footnote{For instance, if one takes the solution $\phi = \phi^{(p)}$ to the linear wave equation $-\phi_{tt}+\phi_{xx} = p 1_{\R \backslash [-1,1]}(x) \phi$ with initial data smooth and supported on $[-1,1]$, a simple compactness argument (or explicit computation) shows that $\phi$ converges (in, say, the uniform topology) to the solution to the free wave equation $-\phi_{tt}+\phi_{xx}=0$ on $\R \times [-1,1]$ with the reflective (Dirichlet) boundary conditions $\phi(t,\pm 1) = 0$.}, one might wish to interpret the infinite nonlinearity $|\phi|^\infty \phi$ as some sort of ``barrier nonlinearity'' which is constraining $\phi$ to have magnitude at most $1$, but otherwise has no effect.  Intuitively, we thus expect the limiting wave $\phi^{(\infty)}$ to evolve like the linear wave equation until it reaches the threshold $\phi^{(\infty)} = +1$ or $\phi^{(\infty)} = -1$, at which point it should ``reflect'' off the nonlinear barrier\footnote{Since each of the equations \eqref{nlw} are Hamiltonian, it is reasonable to expect that \eqref{nlw-infty} should also be ``Hamiltonian'' in some sense (although substituting $p=+\infty$ in \eqref{energy} does not directly make sense), and so energy should be reflected rather than absorbed by the barrier.}.  The purpose of this paper is to make the above intuition rigorous, and to give a precise interpretation for the equation \eqref{nlw-infty}.

\subsection{An ODE analogy}

To get some further intuition as to this reflection phenomenon, let us first study (non-rigorously) the simpler ODE problem, in which we look at solutions $\phi = \phi^{(p)}: \R \to \R$ to the ODE
\begin{equation}\label{ode}
 -\phi_{tt} = |\phi|^{p-1} \phi
 \end{equation}
with fixed initial data $\phi(0) = \phi_0, \phi_t(0) = \phi_1$ with $|\phi_0| \leq 1$, and with $p \to \infty$.  From the conserved energy $\frac{1}{2} \phi_t^2 + \frac{1}{p+1} |\phi|^{p+1}$ (and recalling that $(p+1)^{1/(p+1)} = 1 + \frac{\log p}{p} + O\left(\frac{1}{p}\right)$) we quickly obtain the uniform bounds
\begin{equation}\label{phito}
|\phi_t(t)| = O(1); \quad |\phi(t)| \leq 1 + \frac{\log p}{p} + O\left( \frac{1}{p} \right)
\end{equation}
for all $p$ and all times $t$, where the implied constants in the $O()$ notation depend on $\phi_0, \phi_1$.  Thus we already see a barrier effect preventing $\phi$ from going too far outside of the interval $[-1,1]$.  To investigate what happens near a time $t_0$ in which $\phi(t_0)$ is close to (say) $+1$, let us make the ansatz
$$ \phi(t) = p^{1/(p-1)} (1 + \frac{1}{p} \psi( p(t-t_0) )).$$
Observe from \eqref{phito} that $\phi(t)$ is positive for $|t-t_0| \leq c$ and some constant $c > 0$ depending only on $\phi_0,\phi_1$.
Write $s := p(t-t_0)$.  Some brief computation then shows that $\psi$ solves the equation
$$ \psi_{ss} = -(1 + \frac{1}{p} \psi)^p$$
for all $s \in [-cp,cp]$; also, by \eqref{phito} we obtain an upper bound $\psi \leq O(1)$ (but no comparable lower bound), as well as the Lipschitz bound $|\psi_s| = O(1)$.  In the asymptotic limit $p \to \infty$, we thus expect the rescaled solution $\psi = \psi^{(p)}$ to converge to a solution $\psi = \psi^{(\infty)}$ of the ODE
$$ \psi_{ss} = - e^\psi.$$
It turns out that this ODE can be solved explicitly\footnote{Alternatively, one can reach the desired qualitative conclusions by tracking the ODE along the energy surfaces $\frac{1}{2} \psi_s^2 + e^\psi = \operatorname{const}$ in phase space.}, and it is easy to verify that the general solution is
\begin{equation}\label{ode-sol}
 \psi(s) = \log \frac{2a^2}{\cosh^2(a(s-s_0))}
\end{equation}
for any $s_0 \in \R$ and $a > 0$.  These solutions asymptotically approach $-a|s-s_0| + \log 8a^2$ as $s \to \pm \infty$.  Thus we see that if $\psi$ is large and negative but with positive velocity, then the solution to this ODE will be approximately linear until $\psi$ approaches the origin, where it will dwell for a bounded amount of time before reflecting back into the negative axis with the opposite velocity to its initial velocity.  Undoing the rescaling, we thus expect the limit $\phi = \phi^{(\infty)}$ of the original ODE solutions $\phi^{(p)}$ to also behave linearly until reaching $\phi=+1$ or $\phi=-1$, at which point they should reflect with equal and opposite velocity, so that $\phi^{(\infty)}$ will eventually be a sawtooth function with range $[-1,1]$ (except of course in the degenerate case $\phi_1=0$, $|\phi_0| < 1$, in which case $\phi^{(\infty)}$ should be constant).  Because the ODE can be solved more or less explicitly using the conserved Hamiltonian, it is not difficult to formalise these heuristics rigorously; we leave this as an exercise to the interested reader.  Note that the above analysis also suggests a more precise asymptotic for how reflections of $\phi^{(p)}$ should behave for large $p$, namely (assuming $s_0=0$ for simplicity)
$$ \phi^{(p)}(t) \approx p^{1/(p-1)} \left(1 + \frac{1}{p} \log \frac{2a^2}{\cosh^2(ap(t-t_0))} \right)$$
or (after Taylor expansion)
\begin{equation}\label{phip}
 \phi^{(p)}(t) \approx 1 + \frac{\log p}{p} + \frac{1}{p} \log \frac{2a^2}{\cosh^2(ap(t-t_0))} ,
\end{equation}
where $a$ measures the speed of the reflection, and $t_0$ the time at which reflection occurs, and we are deliberately being vague as to what the symbol $\approx$ means.

Adapting the above ODE analysis to the PDE setting, we can now study the reflection behaviour of $\phi^{(p)}$ near the nonlinear barrier $\phi^{(p)} = 1$ at some point $(t_0,x_0)$ in spacetime by introducing the ansatz
$$ \phi(t,x) = p^{1/(p-1)} (1 + \frac{1}{p} \psi( p(t-t_0), p(x-x_0) )).$$
where $\psi$ can be computed to solve the equation
$$
- \psi_{tt} + \psi_{xx} = -(1 + \frac{1}{p} \psi)^p$$
in the region where $\phi$ is near $1$ (and is in particular non-negative).  In the limit $p \to \infty$, this formally converges to \emph{Liouville's equation}
\begin{equation}\label{psitx}
- \psi_{tt} + \psi_{xx} = e^\psi.
\end{equation}
Remarkably, this nonlinear wave equation can also be solved explicitly\cite{liouville}, with explicit solution
\begin{equation}\label{pde-sol}
\psi = \log \frac{-8 f'(t+x) g'(t-x)}{(f(t+x) + g(t-x))^2}
\end{equation}
for arbitrary smooth functions $f, g$ for which the right-hand side is well-defined\footnote{For instance, the ODE solutions \eqref{ode-sol} can be recovered by setting $f(u) := e^{a(u-t_0)}$ and $g(v) := e^{-a(v-t_0)}$.}.  Somewhat less ``magically'', one can approach the explicit solvability of this equation by introducing the null coordinates
\begin{equation}\label{null-coord}
u := t+x; \quad v := t-x 
\end{equation}
and their associated derivatives
\begin{equation}\label{null}
\partial_u := \frac{1}{2} (\partial_t + \partial_x); \quad \partial_v := \frac{1}{2} (\partial_t - \partial_x)
\end{equation}
and rewriting \eqref{psitx} as
\begin{equation}\label{psiexp}
\psi_{uv} = - \frac{1}{4} e^\psi
\end{equation}
and then noting the \emph{pointwise} conservation laws
\begin{equation}\label{cons}
 \partial_v (\frac{1}{2} \psi_u^2 - \psi_{uu} ) = \partial_u (\frac{1}{2} \psi_v^2 - \psi_{vv} ) = 0
\end{equation}
which can ultimately (with a certain amount of algebraic computation) be used to arrive at the solution \eqref{pde-sol}; see \cite{tao:post} for details.  Using this explicit solution, one can eventually be led to the (heuristic) conclusion that the reflection profile $\psi^{(\infty)}$ should resemble a Lorentz-transformed version of \eqref{ode-sol}, i.e.
\begin{equation}\label{lorentz}
\psi(t,x) = \log \frac{2a^2}{\cosh^2(a[(t-t_0)-v(x-x_0)]/\sqrt{1-v^2})}
\end{equation}
for some $t_0,x_0 \in\R$, $a > 0$, and $-1 < v < 1$.  Thus we expect $\phi$ to reflect along spacelike curves such as $(t-t_0) - v(x-x_0) = 0$ in order to stay confined to the interval $[-1,1]$.

\subsection{Main result} 

We now state the main result of our paper, which aims to make the above intuition precise.

\begin{theorem}[Convergence as $p \to \infty$]\label{main-thm}  Let $\phi_0, \phi_1: \R \to \R$ be functions obeying the following properties:
\begin{itemize}
\item[(a)] (Regularity) $\phi_0, \phi_1$ are smooth.
\item[(b)] (Strict barrier condition) For all $x \in \R$, $|\phi_0(x)| < 1$.
\item[(c)] (Non-degeneracy) The sets $\{ x: \frac{1}{2} (\phi_1 + \partial_x \phi_0)(x) = 0 \}$ and $\{ x: \frac{1}{2} (\phi_1 - \partial_x \phi_0)(x) = 0\}$ have only finitely many connected components in any compact interval\footnote{This condition is automatic if $\phi_0, \phi_1$ are real analytic, since the zeroes of non-trivial real analytic functions cannot accumulate.}.
\end{itemize}
For each $p > 1$, let $\phi^{(p)}: \R \times \R \to \R$ be the unique global solution to \eqref{nlw} with initial data $\phi^{(p)}(0) = \phi_0$, $\phi^{(p)}_t(0) = \phi_1$.  Then, as $p \to \infty$, $\phi^{(p)}$ converges uniformly on compact subsets of $\R \times \R$ to the unique function $\phi = \phi^{(\infty)}: \R \times \R \to \R$ that obeys the following properties:
\begin{itemize}
\item[(i)] (Regularity, I) $\phi$ is locally Lipschitz continuous (and in particular is differentiable almost everywhere, by Radamacher's theorem).
\item[(ii)] (Regularity, II) For each $u \in \R$ and $v \in \R$, the functions $t \mapsto \phi(t, u-t)$ and $t \mapsto \phi(t, t-v)$ are piecewise smooth (with finitely many pieces on each compact interval).
\item[(iii)] (Initial data) On a neighbourhood of the initial surface $\{ (0,x):x \in \R\}$, $\phi$ agrees with the linear solution
\begin{equation}\label{lin}
\phi^{(\lin)}(t,x) := \frac{1}{2} (\phi_0(x+t) + \phi_0(x-t)) + \frac{1}{2} \int_{x-t}^{x+t} \phi_1(y)\ dy
\end{equation}
the free wave equation with initial data $\phi_0, \phi_1$.
\item[(iv)] (Barrier condition) $|\phi(t,x)| \leq 1$ for all $t,x$.
\item[(v)] (Defect measure)  We have
\begin{equation}\label{defect}
-\phi_{tt} + \phi_{xx} = \mu_+ - \mu_-
\end{equation}
in the sense of distributions, where $\mu_+, \mu_-$ are locally finite non-negative measures supported on the sets $\{ (t,x): \phi(t,x) = +1 \}$, $\{ (t,x): \phi(t,x) = -1 \}$ respectively.
\item[(vi)] (Null energy reflection)  For almost every $(t,x)$, we have\footnote{Of course, we can compute the derivatives of $\phi^{(\lin)}$ explicitly from \eqref{lin} in terms of the initial data as $\phi^{(\lin)}_u(t,x) = \frac{1}{2} (\phi_1(0,x+t) + \partial_x \phi_0(0,x+t))$ and $\phi^{(\lin)}_v(t,x) = \frac{1}{2} (\phi_1(0,x-t) - \partial_x \phi_0(0,x-t))$.}
\begin{equation}\label{phiu}
|\phi_u(t,x)| = |\phi^{(\lin)}_u(t,x)|
\end{equation} 
and
\begin{equation}\label{phiv}
|\phi_v(t,x)| = |\phi^{(\lin)}_v(t,x)|
\end{equation} 
In particular, $|\phi_u|$ is almost everywhere equal to a function of $u$ only, and similarly for $|\phi_v|$.
\end{itemize}
\end{theorem}

\begin{remark} The existence and uniqueness of $\phi$ obeying the above properties is not obvious, but is part of the theorem.  The conditions (i)-(vi) are thus the rigorous substitute for the non-rigorous equation \eqref{nlw-infty}; they superficially resemble a ``viscosity solution'' or ``kinetic formulation'' of \eqref{nlw-infty} (see e.g. \cite{perthame}), and it would be interesting to see if there is any rigorous connection here to the kinetic theory of conservation laws.
\end{remark}

\begin{remark} The hypotheses (a), (b), (c) on the initial data $\phi_0, \phi_1$ are somewhat stronger than what is likely to be needed for the theorem to hold; in particular, one should be able to relax the strict barrier condition (b) to $|\phi_0(x)| \leq 1$, and also omit the non-degeneracy condition (c), although the conclusions (ii), (iii) the limit $\phi$ would have to be modified in this case; one also expects to be able to relax the smoothness assumption (a), perhaps all the way to the energy class or possibly even the bounded variation class.  We will not pursue these matters here.
\end{remark}

\subsection{An example}

\begin{figure}[tb]
\centerline{\psfig{figure=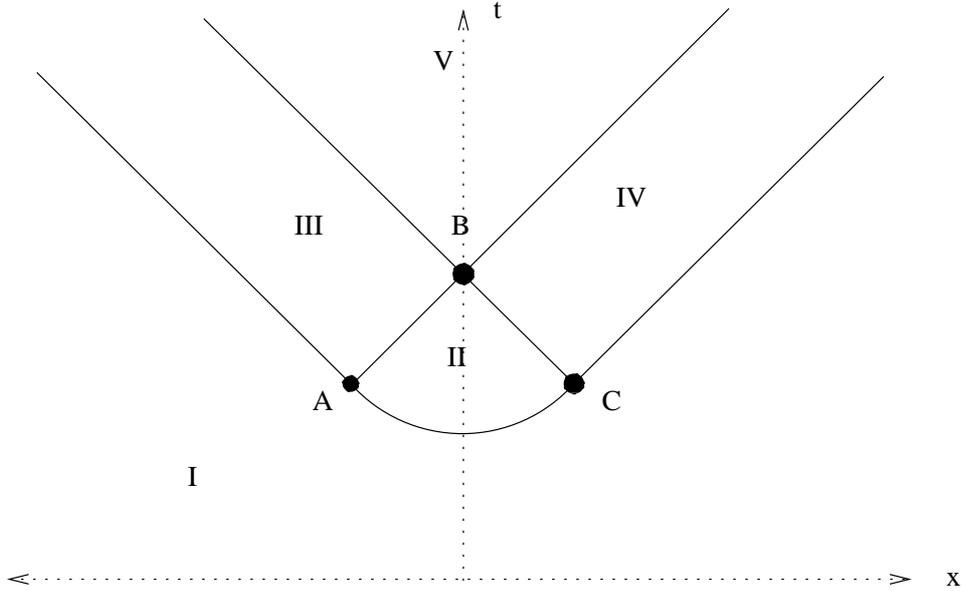}}
\caption{A subdivision of the triangular region $\Delta$ (the diagonal boundaries of $\Delta$ are beyond the scale of the figure).  We have $A = (2-\frac{1}{\sqrt{2}}, -\frac{1}{\sqrt{2}})$, $B = (2,0)$, and $C = (2-\frac{1}{\sqrt{2}},\frac{1}{\sqrt{2}})$.  The circular arc from $A$ to $C$ is part of the circle $\{ (t,x): (t-2)^2 + x^2 = 1 \}$.  The rays bounding regions III and IV are all null rays.}
\label{fig1}
\end{figure}

To illustrate the reflection in action, let us restrict attention to the triangular region
$$ \Delta := \{ (t,x): t \geq 0; |t-x|, |t+x| \leq 10 \}$$
and consider the initial data $\phi_0, \phi_1$ associated to the linear solution
$$ \phi^{(\lin)}(t,x) = 1 - \delta ( (t-2)^2 + x^2 - 1 ) $$
where $\delta > 0$ is a small constant (e.g. $\delta = 10^{-3}$ is safe).  Observe that $\phi^{(\lin)}$ lies between $-1$ and $1$ for most of $\Delta$, but exceeds $1$ in the disk $\{ (t,x): (t-2)^2 + x^2 < 1 \}$.  Thus we expect $\phi$ to follow $\phi^{(\lin)}$ until it encounters this disk, at which point it should reflect.

The actual solution $\phi$ can be described using Figure \ref{fig1}.  In region I, $\phi$ is equal to the linear solution $\phi^{(\lin)}$.  But $\phi^{(\lin)}$ exceeds $1$ once one passes the circular arc joining $A$ and $C$, and so a reflection must occur in region II; indeed, one has
$$ \phi(t,x) = 2 - \phi^{(\lin)}(t,x) = 1 + \delta ( (t-2)^2 + x^2 - 1 ) $$
in this region.  

Once the solution passes $A$ and $C$, though, it turns out that the downward velocity of the reflected wave is now sufficient to drag $\phi$ off of the singular set\footnote{An inspection of \eqref{lorentz} suggests that the singular set must remain spacelike, thus the timelike portions of the set $\{\phi^{(\lin)} = 1 \}$ (which, in this case, are the left and right arcs of the circle $\{ (t,x): (t-2)^2 + x^2 = 1 \}$) are not used as a reflective set for $\phi$.} $\{\phi=+1\}$ (which, in this example, is the circular arc connecting $A$ and $C$).   Indeed, in region III, we have
$$ \phi(t,x) = 1 + \delta(2(t-2)x - 1)$$
(note that this is the unique solution to the free wave equation that matches up with $\phi$ on regions I, II) and similarly on region IV we have
$$ \phi(t,x) = 1 + \delta(-2(t-2)x - 1).$$
Finally, on region V we have another solution to the free wave equation, which is now on a downward trajectory away from $\phi=+1$:
$$ \phi(t,x) = 1 - \delta( (t-2)^2 + x^2 + 1 ).$$
If one were to continue the evolution of $\phi$ forward in time beyond $\Delta$ (extending the initial data $\phi_0, \phi_1$ suitably), the solution would eventually hit the $\phi=-1$ barrier and reflect again, picking up further singularities propagating in null directions similar to those pictured here.  Thus, while the solution remains piecewise smooth for all time, we expect the number of singularities to increase as time progresses, due to the increasing number of reflections taking place.

It is a routine matter to verify that the solution presented here verifies the properties (i)-(vi) on $\Delta$ (if $\delta$ is sufficiently small), and so is necessarily the limiting solution $\phi$, thanks to Theorem \ref{main-thm} (and the uniqueness theory in Section \ref{unique-sec} below).  We omit the details.

\begin{remark} The circular arc between $A$ and $C$ supports a component of the defect measure $\mu_+$, which can be computed explicitly from the above formulae.  The defect measure can also be computed by integrating \eqref{defect} and observing that
$$
\phi\langle u, v \rangle - \phi\langle u-a, v \rangle - \phi\langle u, v-b \rangle + \phi\langle u-a, v-b \rangle = -\mu_+( \{ (u',v'): u-a \leq u' \leq u; v-b \leq v' \leq v \})$$
whenever $\langle u,v \rangle$, $\langle u-a,v\rangle$, $\langle u,v-b\rangle$, $\langle u-a,v-b\rangle$ are the corners of a small parallelogram intersecting this arc.  Sending $b \to 0$ (say), we observe that the left-hand side is asymptotic to $b(\phi_v \langle u, v \rangle - \phi_v \langle u-a, v \rangle)$.  Since $\phi_v$ reflects in sign across the arc, we can simplify this as $b |\phi^{(\lin)}_v\langle u,v\rangle|$.  This allows us to describe $\mu_+$ explicitly in terms of the conserved quantity $|\phi^{(\lin)}_v|$ and the slope of the arc; we omit the details.
\end{remark}

\begin{remark} The above example shows that the barrier set $|\phi|=1$ has some overlap initially with the set $|\phi^{(\lin)}|=1$, but the situation becomes more complicated after multiple ``reflections'' off of the two barriers $\phi=+1$ and $\phi=-1$, and the author does not know of a clean way to describe this set for large times $t$, although as the above example suggests, these sets should be computable for any given choice of $t$ and any given initial data.
\end{remark}

\subsection{Proof strategy}

We shall shortly discuss the proof of Theorem \ref{main-thm}, but let us first pause to discuss two techniques that initially look promising for solving this problem, but end up being problematic for a number of reasons.

Each of the nonlinear wave equations \eqref{nlw} enjoy a conserved stress-energy tensor $T^{(p)}_{\alpha \beta}$, and it is tempting to try to show that this stress-energy tensor converges to a limit $T^{(\infty)}_{\alpha \beta}$.  However, the author found it difficult to relate this limit tensor to the limit solution $\phi^{(\infty)}$.  The key technical difficulty was that while it was not difficult to ensure that derivatives $\phi^{(p)}_u, \phi^{(p)}_v$ of $\phi^{(p)}$ converged in a weak sense to the derivatives $\phi^{(\infty)}_u, \phi^{(\infty)}_v$ of a limit $\phi^{(\infty)}$, this did not imply that the \emph{magnitudes} $|\phi^{(p)}_u|$, $|\phi^{(p)}_v|$ of the derivatives converged (weakly) to the expected limit of $|\phi^{(\infty)}_u|$, $|\phi^{(\infty)}_v|$, due to the possibility of increasing oscillations in the sign of $\phi^{(p)}_u$ or $\phi^{(p)}_v$ in the limit $p \to \infty$ which could cause some loss of mass in the limit.  Because of this, much of the argument is instead focused on controlling this oscillation, and the stress-energy tensor conservation appears to be of limited use for such an objective.  Instead, the argument relies much more heavily on pointwise conservation (or almost-conservation) laws such as \eqref{cons}, and on the method of characteristics.

Another possible approach would be to try to construct an approximate solution (or parametrix) to $\phi^{(p)}$, along the lines of \eqref{phip}, and show that $\phi^{(p)}$ is close enough to the approximate solution that the convergence can be read off directly (much as it can be from \eqref{phip}).  While it does seem possible to construct the approximate solution more or less explicitly, the author was unable to find a sufficiently strong stability theory to then close the argument by comparing the exact solution to the approximate solution.  The difficulty is that the standard stability theory for \eqref{nlw} (e.g. by applying energy estimates to the difference equation) exhibits losses which grow exponentially in time with rate proportional to $p$, thus requiring the accuracy of the approximate solution to be exponentially small in $p$ before there is hope of connecting the approximate solution to the exact one.  Because of this, the proof below avoids all use of perturbation theory\footnote{Except, of course, for the fact that perturbation theory is used to establish global existence of the $\phi^{(p)}$.}, and instead estimates the nonlinear solutions $\phi^{(p)}$ directly. It may however be of interest to develop a stability theory for \eqref{nlw} which is more uniform in $p$ (perhaps using bounded variation type norms rather than energy space norms?).   One starting point may be the perturbation theory for \eqref{psitx}, explored recently in \cite{kal}.

Our arguments are instead based on a compactness method.  It is not difficult to use energy conservation to demonstrate equicontinuity and uniform boundedness in the $\phi^{(p)}$, so we know (from the Arzel\'a-Ascoli theorem) that the $\phi^{(p)}$ have at least one limit point.  It thus suffices to show that all such limit points obey the properties (i)-(vi), and that the properties (i)-(vi) uniquely determine $\phi$.  The uniqueness is established in Section \ref{unique-sec}, and is based on many applications of the method of characteristics.  To establish that all limit points obey (i)-(vi), we first establish in Section \ref{apriori} a number of \emph{a priori} estimates on the solutions $\phi^{(p)}$, in particular obtaining some crucial boundedness and oscillation control on $\phi$ and its first derivatives, uniformly in $p$.  In Section \ref{compact-sec} we then take limits along some subsequence of $p$ going to infinity to recover the desired properties (i)-(vi).

\begin{remark} It seems of interest to obtain more robust methods for proving results for infinite nonlinear barriers; the arguments here rely heavily on the method of characteristics and so do not seem to easily extend to, say, the $p \to \infty$ limit of the one-dimensional nonlinear Schr\"odinger equation $iu_t + u_{xx} = |u|^{p-1} u$, or to higher-dimensional non-linear wave equations.  In higher dimensions there is also a serious additional problem, namely that the nonlinearity becomes energy-critical in the limit $p \to \infty$ in two dimensions, and (even worse) becomes energy-supercritical for large $p$ in three and higher dimensions.  However, while global existence for defocusing supercritical non-linear wave equations from large data is a notoriously difficult open problem, there is the remote possibility that the asymptotic case $p \to \infty$ is actually better behaved than that of a fixed $p$.  At the very least, one should be able to \emph{conjecture} what the correct limit of the solution should be.
\end{remark}

\subsection{Acknowledgments}

We are indebted to Tristan Roy for posing this question, and Rowan Killip for many useful discussions, Ut V. Le for pointing out a misprint, to Patrick Dorey to pointing out the link to Liouville's equation, and to the anonymous referee for many useful comments and corrections. The author is supported by a grant from the MacArthur Foundation, by NSF grant DMS-0649473, and by the NSF Waterman award.

\section{Notation}

We use the asymptotic notation $X \ll Y$ to denote the bound $X \leq CY$ for some constant $C$ depending on fixed quantities (e.g. the initial data); note that $X$ may be negative, so that $X \ll Y$ only provides an upper bound.  We also use $O(X)$ to denote any quantity bounded in magnitude by $CY$ (thus we have both an upper and a lower bound in this case), and $X \sim Y$ for $X \ll Y \ll X$.  If the constant $C$ needs to depend on additional parameters, we will denote this by subscripts, e.g. $X \ll_\eps Y$ or $X = O_\eps(Y)$.

It is convenient to use both Cartesian coordinates $(t,x)$ and null coordinates $\langle u,v\rangle$ to parameterise spacetime.  To reduce confusion we shall use angled brackets to denote the latter, thus
$$ (t,x) = \langle t+x, t-x \rangle$$
and
$$ \langle u, v \rangle = \left( \frac{u+v}{2}, \frac{u-v}{2} \right).$$
Thus for instance we might write $\phi\langle u,v\rangle$ for $\phi( \frac{u+v}{2}, \frac{u-v}{2} )$.  

We will frequently rely on three reflection symmetries of \eqref{nlw} to normalise various signs: the \emph{time reversal symmetry}
\begin{equation}\label{time-reverse}
\phi(t,x) \mapsto \phi(-t,x)
\end{equation}
(which also swaps $u$ with $-v$), the \emph{space reflection symmetry}
\begin{equation}\label{space-reverse}
\phi(t,x) \mapsto \phi(t,-x)
\end{equation}
(which also swaps $u$ with $v$), and the \emph{sign reversal symmetry}
\begin{equation}\label{sign-reverse}
\phi(t,x) \mapsto -\phi(t,x).
\end{equation}

We will frequently be dealing with (closed) \emph{diamonds} in spacetime, which we define to be regions of the form 
$$\{ \langle u,v\rangle: u_0-r \leq u \leq u_0; v_0-r \leq v \leq v_0 \}$$
for some $u_0,v_0 \in \R$ and $r>0$.  One can of course define open diamonds similarly.  We will also be dealing with \emph{triangles}
$$\{ \langle u,v\rangle: u_0-r \leq u \leq u_0; v_0-r \leq v \leq v_0; \quad u+v \geq u_0+v_0-r \}$$
which are the upper half of diamonds.

\section{Uniqueness}\label{unique-sec}

In this section we show that there is at most one function $\phi: \R \to \R$ obeying the properties (i)-(vi) listed in Theorem \ref{main-thm}.  It suffices to prove uniqueness on a diamond region $\{ \langle u, v \rangle: |u|, |v| \leq T \}$ for any fixed $T > 0$; by the time reversal symmetry \eqref{time-reverse} it in fact suffices to prove uniqueness on a triangular region
$$ \Delta := \{ \langle u,v\rangle = (t,x): |u|, |v| \leq T; t \geq 0 \}.$$
Suppose for contradiction that uniqueness failed on $\Delta$, then there would exist two functions $\phi, \phi': \Delta \to \R$ obeying the properties (i)-(vi) in Theorem \ref{main-thm} which did not agree identically on $\Delta$.  

By property (iii), $\phi$ and $\phi'$ already agree on some neighbourhood of the time axis.  Since $\phi, \phi'$ are continuous by (i), and $\Delta$ is compact, we may therefore find some $(t_0,x_0) \in \Delta$ with $0 < t_0 < T$ such that $\phi(t,x) = \phi'(t,x)$ for all $(t,x) \in \Delta$ with $t \leq t_0$, but such that $\phi$ and $\phi'$ do not agree in any neighbourhood of $(t_0,x_0)$ in $\Delta$.  We will show that $\phi$ and $\phi'$ must in fact agree in some neighbourhood of $(t_0, x_0)$, achieving the desired contradiction.

First suppose that $|\phi(t_0,x_0)| < 1$, then of course $|\phi'(t_0,x_0)| < 1$.  As $\phi, \phi'$ are both continuous, we thus have $\phi, \phi'$ bounded away from $-1$ and $+1$ on a neighbourhood of $(t_0,x_0)$ in $\Delta$.  By (v), $\phi, \phi'$ both solve the free wave equation in the sense of distributions on this region, and since they agree below $(t_0,x_0)$, they must therefore agree on a neighbourhood of $(t_0,x_0)$ by uniqueness of the free wave equation, obtaining the desired contradiction.  Thus we may assume that $\phi(t_0,x_0)=\phi'(t_0,x_0)$ has magnitude $1$; by the sign reversal symmetry \eqref{sign-reverse} we may take 
$$ \phi(t_0,x_0) = \phi'(t_0,x_0) = +1.$$ 
By continuity, we thus see that $\phi$, $\phi'$ is positive in a neighbourhood of $(t_0,x_0)$; from (v) we conclude that $-\phi_{tt} + \phi_{xx}$ is a non-negative measure in this neighbourhood.  Integrating this, we conclude that
\begin{equation}\label{rect}
\phi\langle u, v \rangle - \phi\langle u-a, v \rangle - \phi\langle u, v-b \rangle + \phi\langle u-a, v-b \rangle \leq 0
\end{equation}
whenever $a,b > 0$ and $\langle u,v \rangle, \langle u-a, v \rangle, \langle u, v-b \rangle, \langle u-a, v-b \rangle \in \Delta$ lie sufficiently close to $(t_0,x_0)$; this implies in particular that $\phi_u$ is non-increasing in $v$ (and $\phi_v$ non-increasing in $u$) in this region whenever the derivatives are defined.  Similarly for $\phi'$.

\subsection{Extension to the right}

Write $\langle u_0,v_0\rangle := (t_0,x_0)$.  We already know that $\phi\langle u,v\rangle = \phi'\langle u,v\rangle$ when $\langle u,v\rangle$ is sufficiently close to $\langle u_0, v_0\rangle$ and $u+v \leq u_0+v_0$.  We now make extend this equivalence to the right of $u_0,v_0$:

\begin{lemma}\label{phiright} $\phi$ and $\phi'$ agree for all $\langle u, v\rangle$ sufficiently close to $\langle u_0,v_0\rangle$ in the region $u \geq u_0$, $v \leq v_0$.  
\end{lemma}

\begin{proof}  We may of course assume $u_0<T$ otherwise this extension is vacuous.

\begin{figure}[tb]
\centerline{\psfig{figure=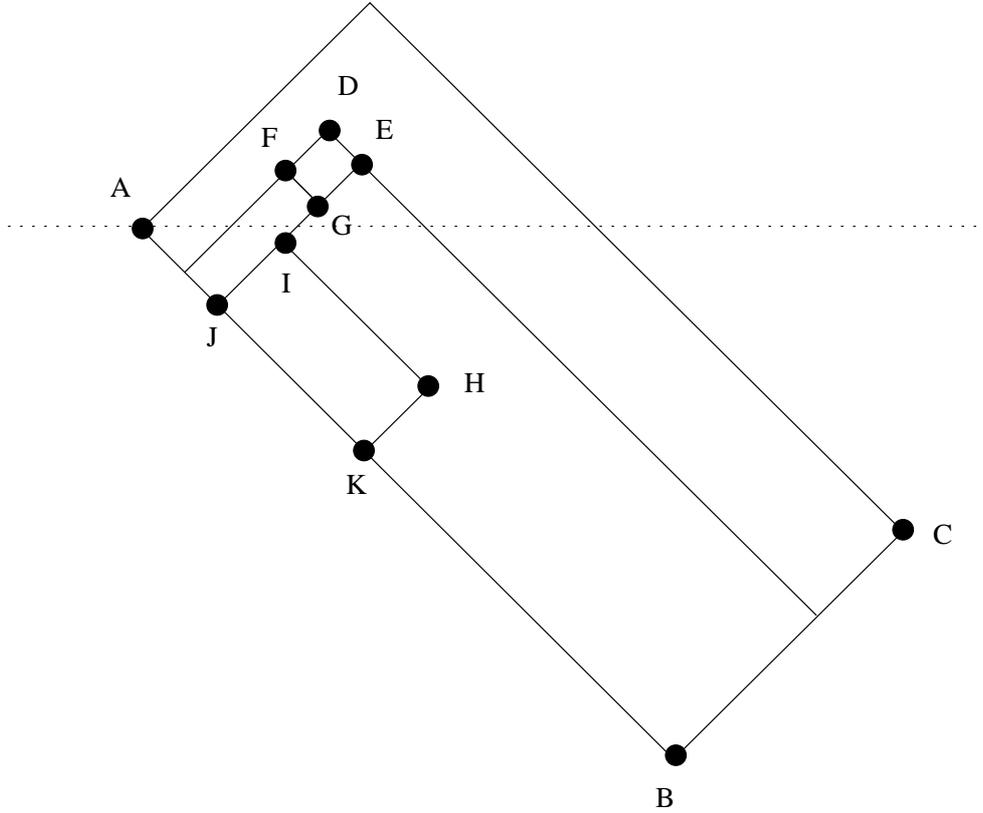}}
\caption{The proof of Lemma \ref{phiright} in the most difficult case.  $A = \langle u_0,v_0\rangle$, $B = \langle u_0,v_0-b_0\rangle$, $C = \langle u_0+a_0,v_0-b_0\rangle$, $D = \langle u, v \rangle$, $E = \langle u, v' \rangle$, $F  = \langle u', v \rangle$, $G = \langle u', v' \rangle$, $H = \langle \tilde u, \tilde v\rangle$, $I = \langle \tilde u, v' \rangle$, $J = \langle u_0, v' \rangle$, $K = \langle u_0, \tilde v \rangle$.  $\phi$ and $\phi'$ are already known to agree below the dotted line; the task is to extend this agreement to the parallelogram $P$ with $A,B,C$ as three of its corners. $\phi$ is also known in this case to be strictly decreasing from $A$ to $B$ and strictly increasing from $B$ to $C$.}
\label{fig2}
\end{figure}

Suppose first that $\phi\langle u_0, v_0-b\rangle = 1$ for a sequence of positive $b$ approaching zero.  Since $a \mapsto \phi\langle u_0,v_0-a\rangle$ is piecewise smooth by (ii), we see from the mean value theorem that $\phi_v\langle u_0, v_0-b'\rangle=0$ for a sequence of positive $b'$ approaching zero; by (vi) the same is true for $\phi^{(\lin)}$, which by the non-degeneracy assumption (c) implies that $\phi^{(\lin)}_v$ must in fact vanish on some left-neighbourhood of $v_0$, which by (vi) implies that $\phi_v\langle u,v\rangle$ and $\phi'_v\langle u, v\rangle$ vanish almost everywhere whenever $v$ is less than $v_0$ and sufficiently close to $v_0$.  Since $\phi, \phi'$ already agree for $u+v \leq u_0+v_0$, and are Lipschitz continuous, an integration in the $v$ direction then gives the lemma.

Now consider the opposite case, where $\phi\langle u_0, v_0-b\rangle < 1$ for all sufficiently small positive $b$.  As $b \mapsto \phi\langle u_0, v_0-b \rangle$ is piecewise smooth, and its stationary points have finitely many connected components in $\Delta$ by (vi), (c), we conclude that $\phi_v\langle u_0, v-b \rangle > 0$ for all $0 < b < b_0$ for some sufficiently small $b_0$.

By continuity, we can find a small $a_0 \in (0,b_0)$ such that $\phi\langle u_0+a, v-b_0\rangle < 1$ for all $0 < a < a_0$.  Since $a \mapsto \phi\langle u_0+a, v-b_0\rangle$ is piecewise smooth and the stationary points have finitely many connected components in $\Delta$, we see (after shrinking $a_0$ if necessary) that $\phi_u\langle u_0+a, v-b_0\rangle$ is either always positive, always negative, or always zero for $0 < a < a_0$.

If $\phi_u\langle u_0+a, v-b_0\rangle$ is always zero, then by (vi) we see that $\phi_u\langle u, v\rangle=\phi'_u\langle u, v\rangle=0$ almost everywhere for $u$ sufficiently close to and larger than $u_0$, and any $v$, and the lemma then follows from the fundamental theorem of calculus.

If $\phi_u\langle u_0+a, v-b_0\rangle$ is always negative, then as $\phi_u$ is non-increasing in $v$, we see that $\phi_u\langle u_0+a, v-b\rangle$ is negative for almost every $0 < a < a_0$ and $-b_0 < b < b_0$; combining this with (vi) we conclude that $\phi_u\langle u_0+a, v-b\rangle=\phi_u\langle u_0+a, v-b_0\rangle$ almost everywhere in this region, and similarly for $\phi'$.  In particular, $\phi_u$ and $\phi'_u$  agree in this region, and the lemma again follows from the fundamental theorem of calculus.

Finally, we handle the most difficult\footnote{Indeed, this is the one case where $\phi$ will reflect itself on some spacelike curve containing $(t_0,x_0)$, which is essentially the only interesting nonlinear phenomenon that $\phi$ can exhibit.} case, when $\phi_u\langle u_0+a, v-b_0\rangle$ is always positive.  Let $\Phi$ denote the unique Lipschitz-continuous solution to the free wave equation on the parallelogram $P := \{ \langle u, v \rangle: u_0 \leq u \leq u_0+a_0; v_0-b_0 \leq v \leq v_0 \}$ that agrees with $\phi$ (and hence $\phi'$) on the lower two edges of this parallelogram (i.e. when $u=u_0$ or $v=v_0-b_0$).  We will show that $\phi$ and $\phi'$ agree on $P$.

Call a point $\langle u,v\rangle \in P$ \emph{good} if we have $\Phi(u',v') \leq 1$ for all $\langle u',v' \rangle \in P$ with $u' \leq u, v' \leq v$; the set of good points is then a closed subset of $P$.  

We first observe that if $\langle u,v\rangle$ is good, then $\phi=\phi'=\Phi$.  Indeed, since $\phi_v$ is positive on the lower left edge of $P$, and $\phi_u$ is positive on the lower right edge, we see that $\Phi\langle u',v'\rangle < 1$ for all $\langle u',v' \rangle \in P$ with $u' \leq u, v' \leq v$, and $\langle u',v' \rangle \neq \langle u,v \rangle$.  Also, from (v), $\phi$ and $\phi'$ solve the free wave equation in the neighbourhood of any region near $\langle u_0, v_0 \rangle$ where they are strictly less than $1$.  A continuity argument based on the uniqueness of the free wave equation then shows that $\phi\langle u',v'\rangle = \phi'\langle u',v' \rangle = \Phi\langle u',v' \rangle$ for all $\phi_u\langle u_0+a, v-b_0\rangle$, and the claim follows.

Now suppose that $\langle u,v \rangle$ is not good, thus $\Phi\langle u,v\rangle > 1$.  Excluding a set of measure zero, we may assume that $v_0-b_0 < v < v_0$ and that $|\phi_v\langle u,v\rangle|$ exists and is equal to $|\phi_v\langle u_0, v\rangle|$, which is non-zero.  We claim that $\phi_v\langle u, v \rangle$ cannot be positive.  For if it is, then we can find $v'$ less than $v$ and arbitrarily close to $v$ such that
$$ \phi\langle u, v' \rangle < \phi\langle u, v \rangle.$$
(See Figure \ref{fig2}.)  Applying \eqref{rect}, we conclude that
$$ \phi\langle u', v' \rangle < \phi\langle u', v \rangle$$
for all $u_0 \leq u' \leq u$.  In particular, $\phi\langle u',v'\rangle$ must be bounded away from $1$.  Also, as $\Phi$ is continuous, we may also assume that
$$ \Phi\langle u, v' \rangle > 1.$$
In particular, $\phi$ and $\Phi$ disagree at $u, v'$.  This implies that $\phi\langle \tilde u, \tilde v \rangle = 1$ for at least one $u_0 \leq \tilde u \leq u$, $v_0-b_0 \leq \tilde v \leq v'$, since otherwise by (v) $\phi$ would solve the free wave equation in this region and thus be necessarily equal to $\Phi$ by uniqueness of that equation.  Among all such $\langle \tilde u, \tilde v \rangle$, we can (by continuity and compactness) pick a pair that maximises $\tilde v$.  Since $\phi\langle u',v'\rangle$ must be bounded away from $1$, we have $\tilde v < v'$.  From (v), $\phi$ solves the linear wave equation in the parallelogram $\{ \langle u'',v''\rangle: u_0 \leq u'' \leq \tilde u; \tilde v \leq v'' \leq v' \}$, which implies that
$$ \phi\langle \tilde u, v' \rangle - \phi \langle \tilde u, \tilde v \rangle - \phi \langle u_0, v' \rangle + \phi \langle u_0, \tilde v \rangle = 0.$$
But by the fundamental theorem of calculus, $\phi \langle u_0, v' \rangle > \phi \langle u_0, \tilde v \rangle$, and $\phi\langle \tilde u, \tilde v \rangle = 1$, hence $\phi\langle \tilde u, v' \rangle > 1$, contradicting (iv).  Hence $\phi_v\langle u,v \rangle$ cannot be positive, and hence by hypothesis must be equal to $-|\phi_v\langle u_0, v\rangle|$.  The same considerations apply to $\phi'$, and so $\phi_v = \phi'_v$ at almost every point in $P$ that is not good.

Since $\phi=\phi'$ at good points in $P$, and $\phi_v = \phi'_v$ on all other points of $P$, we obtain the lemma from the fundamental theorem of calculus.
\end{proof}

\subsection{Extension to the left}

Using space reflection symmetry \eqref{space-reverse}, we can reflect the previous lemma and conclude

\begin{lemma} $\phi$ and $\phi'$ agree for all $\langle u, v\rangle$ sufficiently close to $\langle u_0,v_0\rangle$ in the region $u \leq u_0$, $v \geq v_0$.  
\end{lemma}

\subsection{Extension to the future}

By the previous lemmas, we know that there exist $a_0, b_0 > 0$ such that $\phi\langle u_0+a,v_0 \rangle = \phi'\langle u_0+a,v_0\rangle$ for all $0 \leq a \leq a_0$ and $\phi\langle u_0,v_0+b \rangle = \phi'\langle u_0,v_0+b\rangle$ for all $0 \leq a \leq a_0$.  To finish the uniqueness claim, we need to show that $\phi, \phi'$ also agree in the future parallelogram $F := \{ \langle u,v\rangle: u_0 \leq u \leq u_0+a_0; v_0 \leq v \leq v_0+b_0 \}$.

As before, by shrinking $a_0$ if necessary, we know that $\phi_u\langle u_0+a,v_0\rangle$ is either always positive, always zero, or always negative for $0 < a < a_0$.  The former option is not possible from the barrier condition (iv) since $\phi\langle u_0,v_0\rangle = 1$, so $\phi_u\langle u_0+a,v_0\rangle$ is always non-positive.  Using the monotonicity of $\phi_u$ in $v$ and (vi) we thus conclude that $\phi_u\langle u_0+a,v_0+b\rangle = \phi_u\langle u_0+a,v_0\rangle$ for almost every $0 < a < a_0, 0 < b < b_0$, and similarly for $\phi'$, so that $\phi_u$ and $\phi'_u$ agree almost everywhere on $F$; similarly $\phi_v$ and $\phi'_v$ agree.  The claim now follows from the fundamental theorem of calculus, and the uniqueness claim is complete.

\begin{remark}  One could convert the above uniqueness results, with additional effort, into an existence result, but existence of a solution to (i)-(vi) will be automatic for us from the Arzel\'a-Ascoli theorem, as we will show that any uniformly convergent sequence of $\phi^{(p)}$ will converge to a solution to (i)-(vi).
\end{remark}

\section{A priori estimates}\label{apriori}

The next step in the proof of Theorem \ref{main-thm} is to establish various \emph{a priori} estimates on the solution to \eqref{nlw} on the diamond
$$ \diamondsuit_{T_0} := \{ \langle u, v \rangle: |u|, |v| \leq T_0 \}$$
for some large parameter $T_0$.  Accordingly, let us fix $\phi_0, \phi_1$ obeying the hypotheses (a),(b),(c) of Theorem \ref{main-thm}, and let $T_0 > 0$; we allow all implied constants to depend on the initial data $\phi_0, \phi_1$ and $T_0$, but will carefully track the dependence of constants on $p$.  We will assume that $p$ is sufficiently large depending on the initial data and on $T_0$; in particular, we may take $p \geq 100$ (say).
Standard energy methods (see e.g. \cite{sogge:wave}) then show that $\phi$ exists globally and is $C^{10}$ in $\diamondsuit_{T_0}$. This is sufficient to justify all the formal computations below.

We begin with a preliminary (and rather crude) H\"older continuity estimate, which we need to establish some spatial separation (uniformly in $p$) between the region where $\phi$ approaches $+1$, and the region where $\phi$ approaches $-1$.

\begin{lemma}[H\"older continuity]\label{hold}  For any $(x_1,t_1), (x_2,t_2) \in \diamondsuit_{T_0}$ we have
$$ |\phi(t_1,x_1) - \phi(t_2,x_2)| \ll |x_1-x_2|^{1/2} + |t_1-t_2|^{1/2}.$$
\end{lemma}

\begin{proof}
We use the monotonicity of local energy
$$ E(t) := \int_{x: (t,x) \in \diamondsuit_{T_0}} \frac{1}{2} \phi_t^2 + \frac{1}{2} \phi_x^2 + \frac{1}{p+1} |\phi|^{p+1}\ dx.$$
The standard energy flux identity shows that $E(t) \leq E(0)$ for all $t$.  From the hypotheses (a), (b) we have
$$ |E(0)| \ll 1$$
(note in particular the uniformity in $p$) and thus by energy monotonicity
\begin{equation}\label{eo}
 |E(t)| \ll 1
\end{equation}
for all $t$.  By Cauchy-Schwarz, we see in particular that we have some H\"older continuity in space, or more precisely that
\begin{equation}\label{phitx}
|\phi(t,x) - \phi(t,x')| \ll |x-x'|^{1/2}
\end{equation}
whenever $(t,x), (t,x') \in \diamondsuit_{T_0}$.  We can also get some H\"older continuity in time by a variety of methods.  For instance, from Cauchy-Schwarz again we have
$$ |\phi(t_1,x) - \phi(t_2,x)|^2 \leq |t_1-t_2| (\int_{t_1}^{t_2} \phi_t^2\ dt)$$
for any $t_1 < t_2$ and any $x$; integrating this in $x$ on some interval $[x_0-r,x_0+r]$ and using \eqref{eo} we obtain
$$ \int_{x_0-r}^{x_0+r} |\phi(t_1,x) - \phi(t_2,x)|^2\ dx \ll |t_1-t_2|^2$$
when $(t_1,x_0-r), (t_1,x_0+r), (t_2,x_0-r), (t_2,x_0+r) \in \diamondsuit_{T_0}$.
On the other hand, from \eqref{phitx} we have
$$ |\phi(t_1,x_0) - \phi(t_2,x_0)|^2 \ll |\phi(t_1,x) - \phi(t_2,x)|^2 + r$$
and thus
$$ 2 r |\phi(t_1,x_0) - \phi(t_2,x_0)|^2 \ll |t_1-t_2|^2 + r^2.$$
If we optimise $r := |t_1-t_2|$, we obtain
$$ |\phi(t_1,x_0) - \phi(t_2,x_0)| \ll |t_1-t_2|^{1/2};$$
combining this with \eqref{phitx} we obtain the claim (possibly after replacing $T_0$ with a slightly larger quantity in the above argument).
\end{proof}

Next, we express the equation \eqref{nlw} in terms of the null derivatives \eqref{null} as
\begin{equation}\label{phiuv}
 \phi_{uv} = - \frac{1}{4} |\phi|^{p-1} \phi.
\end{equation}
We can use this to give some important pointwise bounds on $\phi$ and its derivatives.  For any time $-T_0 \leq t_0 \leq T_0$, let $K(t_0)$ be the best constant such that
\begin{equation}\label{phikk}
 |\phi(t_0,x)| \leq 1 + \frac{\log p}{p} + \frac{K}{p}
\end{equation}
and
\begin{equation}\label{phik}
 |\phi_u(t_0,x)|, |\phi_v(t_0,x)| \leq K
\end{equation}
and
\begin{equation}\label{phikkk}
 |\phi_{uu}(t_0,x)|, |\phi_{vv}(t_0,x)| \leq Kp
\end{equation}
for all $x$ with $(t_0,x) \in \diamondsuit_{T_0}$ (compare with \eqref{phito}).  Thus for instance $K(0) \ll 1$.  We now show that the \eqref{phikk} component of $K(t_0)$, at least, is stable on short time intervals:

\begin{lemma}[Pointwise bound]\label{pond}  There exists a time increment $\tau > 0$ (depending only on the initial data and $T_0$) such that
$$ |\phi(t_1,x_1)| \leq 1 + \frac{\log p}{p} + O_{K(t_0)}( \frac{1}{p} )$$
for any $-T_0 \leq t_0 \leq T_0$ and $(t_1, x_1) \in \diamondsuit_{T_0}$ with $|t_1-t_0| \leq \tau$, and either $t_1 \geq t_0 \geq 0$ or $t_1 \leq t_0 \leq 0$. 
\end{lemma}

\begin{proof}  We shall use the method of characteristics.
We take $\tau > 0$ to be a small quantity depending on the initial data to be chosen later.
Fix $t_0$ and write $K := K(t_0)$.  Let $\epsilon_K > 0$ be a small quantity depending on $K$ and the initial data to be chosen later, and then let $C_K$ be a large quantity depending on $K,\epsilon_K > 0$ to be chosen later.
By time reversal symmetry \eqref{time-reverse} we may take $t_1 \geq t_0 \ge 0$; by sign reversal symmetry \eqref{sign-reverse} it suffices to establish the upper bound
$$ \phi(t_1,x_1) \leq 1 + \frac{\log p}{p} + \frac{C_K}{p}.$$

Assume this bound fails, then (by continuity and compact support in space) there exists $t_0 \leq t_1  \leq t_0 + \tau$ and $x_1 \in \R$ such that
\begin{equation}\label{phi00}
\phi(t_1,x_1) = 1 + \frac{\log p}{p} + \frac{C_K}{p} 
\end{equation}
and
\begin{equation}\label{phi01}
 \phi(t,x) \leq 1 + \frac{\log p}{p} + \frac{C_K}{p} 
\end{equation}
for all $t_0 \leq t < t_1$ and $x \in \R$.  From \eqref{phikk}, we have $t_1 > t_0$ if $C_K$ is large enough.

From Lemma \ref{hold} we see (if $\tau$ is small enough) that $\phi(t,x)$ is positive in the triangular region 
$$\Delta := \{ (t, x): t_0 \leq t \leq t_1; |x-x_1| \leq |t-t_1| \}$$
(which is contained in $\diamond_{T_0}$). From \eqref{phiuv} we conclude that $\phi_u$ is decreasing in the $v$ direction in $\Delta$, and thus has an upper bound $\phi_u \leq K$ on this region thanks to \eqref{phik}.  Similarly we have $\phi_v \leq K$ on $\Delta$.  Applying the fundamental theorem of calculus and \eqref{phikk} we conclude that
$$\phi(t,x) \leq 1 + \frac{\log p}{p} + \frac{K}{p} + O( K (t_1-t_0) ),$$
for all $(t,x) \in \Delta$, which when compared with \eqref{phi00} shows (if $\epsilon_K$ is small enough and $C_K$ is large enough) that
$$ t_1 \geq t_0 + 2r$$
where $r := \frac{\epsilon_K C_K}{p}$.  

Now we consider the diamond region
$$ \diamondsuit := \{ (t,x): t_1+x_1-r\leq t+x \leq t_1+x_1; \quad t_1-x_1-r \leq t-x \leq t_1-x_1\}.$$
Since $t_1 > 2r$, this diamond is contained in the triangle $\Delta$ (indeed, it is nestled in the upper tip of that triangle).  As before, we have the upper bounds
$$ \phi_u, \phi_v \leq K$$
on this diamond.  From this, \eqref{phi00}, and the fundamental theorem of calculus, we have
\begin{equation}\label{philower}
\phi(t,x) \geq 1 + \frac{\log p}{p}
\end{equation}
on this diamond (if $\epsilon_K$ is small enough).  Applying \eqref{phiuv} we conclude that
$$ \phi_{uv} \leq -(1 + \frac{\log p}{p})^p \leq - cp$$
for some absolute constant $c > 0$.  Integrating this on the diamond we conclude that
$$ \phi(t_1,x_1) - \phi(t_1-r,x_1-r) - \phi(t_1-r,x_1+r) + \phi(t_1-2r,x_0) \leq -cp r^2.$$
But from \eqref{phi01}, \eqref{philower}, the left-hand side is bounded below by $-O( C_K / p )$.  We conclude that
$$ p r^2 \ll \frac{C_K}{p}.$$
But from the definition of $r$, we obtain a contradiction if $C_K$ is large enough depending on $\epsilon_K$, and the claim follows.
\end{proof}

Now we establish a similar stability for the \eqref{phik} component of $K(t_0)$.

\begin{lemma}[Pointwise bound for derivatives]\label{ponder}  There exists a time increment $\tau > 0$ (depending only on the initial data and $T_0$) such that
$$ |\phi_u(t_1,x_1)|, |\phi_v(t_1,x_1)| \ll_{K(t_0)} 1 $$
and
$$ |\phi_{uu}(t_1,x_1)|, |\phi_{vv}(t_1,x_1)| \ll_{K(t_0)} p $$
for any $-T_0 \leq t_0 \leq T_0$ and $(t_1, x_1) \in \diamondsuit_{T_0}$ with $|t_1-t_0| \leq \tau$, and either $t_1 \geq t_0 \geq 0$ or $t_1 \leq t_0 \leq 0$. 
\end{lemma}

\begin{proof}  This will be a more advanced application of the method of characteristics.  We again let $\tau > 0$ be a sufficiently small quantity (depending on the initial data) to be chosen later.  Fix $t_0$ and write $K=K(t_0)$, and let $C_K > 0$ be a large quantity depending on $K$ and the initial data to be chosen later.

It will suffice to show that
$$ |\phi_u(t_1,x_1)|^2 + \frac{1}{p} |\phi_{uu}(t_1,x_1)|, |\phi_v(t_1,x_1)|^2 + \frac{1}{p} |\phi_{vv}(t_1,x_1)| \leq C_K$$
whenever $|t_1-t_0| \leq \tau$ and $x_1 \in \R$.

Suppose for contradiction that this claim failed. As before (using the symmetries \eqref{time-reverse}, \eqref{space-reverse}) we may assume that $0 \leq t_0 \leq t_1 \leq t_0+\tau$ and $x_1$ are such that
\begin{equation}\label{phiuc-0}
|\phi_u(t_1,x_1)|^2 + \frac{1}{p} |\phi_{uu}(t_1,x_1)| = C_K
\end{equation}
(say), and that
\begin{equation}\label{phiuc}
 |\phi_u(t,x)|^2 + \frac{1}{p} |\phi_{uu}(t,x)|, |\phi_v(t,x)|^2 + \frac{1}{p} |\phi_{vv}(t,x)| \leq C_K
\end{equation}
for all $t_0 \leq t \leq t_1$ and $x$ with $(t,x) \in \diamondsuit_{T_0}$.

We first dispose of an easy case when $\phi(t_1,x_1)$ is small, say $|\phi(t_1,x_1)| \leq 1/2$.  Then by Lemma \ref{hold} we conclude (if $\tau$ is small enough) that $|\phi| \leq 1$ on the triangular region $\{ (t,x): t_0 \leq t \leq t_1: |x_1-x| \leq |t_1-t| \}$, and the claim then easily follows from \eqref{phiuv}, the fundamental theorem of calculus, and \eqref{phik}.  Thus we may assume $|\phi(t_1,x_1)| > 1/2$; replacing $\phi$ with $-\phi$ if necessary we may assume $\phi(t_1,x_1) > 1/2$.

By Lemma \ref{hold}, we see that $\phi$ is positive whenever $|t-t_1|, |x-x_1| \leq 100 \tau$ (say), if $\tau$ is small enough.

It will be convenient to make the change of variables
$$ \phi(t,x) = p^{1/(p-1)} (1 + \frac{1}{p} \psi( p(t-t_0), p(x-x_1) ))$$
then from \eqref{phiuv}, $\psi$ solves the equation
\begin{equation}\label{psiuv}
\psi_{uv} = - \frac{1}{4} (1 + \frac{1}{p} \psi)^p
\end{equation}
on the region $|t|, |x| \leq 100 \tau p$.  From \eqref{phiuc-0}, \eqref{phiuc} we have
\begin{equation}\label{psipsi0}
 |\psi_u( p(t_1-t_0), 0)|^2 + |\psi_{uu}( p(t_1-t_0), 0)| \sim C_K
\end{equation}
and
\begin{equation}\label{psipsi}
 |\psi_u(t,x)|^2, |\psi_v(t,x)|^2, |\psi_{uu}(t,x)|, |\psi_{vv}(t,x)| \ll C_K.
\end{equation}
for $0 \leq t \leq p(t_1-t_0)$ and $|x| \leq 100 \tau p$.  Meanwhile, while from Lemma \ref{pond} (and shrinking $\tau$ as necessary) we have the upper bound
\begin{equation}\label{psiup}
 \psi(t,x) \ll_K 1
\end{equation}
for $|t|, |x| \leq 100 \tau p$.  Note though that we do not expect $\psi$ to enjoy a comparable lower bound, but since $\phi$ is positive in the region of interest, we have 
\begin{equation}\label{psip}
\psi(t,x) \geq -p
\end{equation}  
for $|t|, |x| \leq 100 \tau p$.
Finally, from \eqref{phik}, \eqref{phikk} we have
\begin{equation}\label{psi0}
 |\psi_u(0,x)|, |\psi_v(0,x)|, |\psi_{uu}(0,x)|, |\psi_{vv}(0,x)| \ll_K 1
 \end{equation}
whenever $|x| \leq 100 \tau p$.

Motivated by the pointwise conservation laws \eqref{cons} of the equation \eqref{psiexp}, which \eqref{psiuv} formally converges to, we consider the quantity
$$ \partial_v ( \frac{1}{2} \psi_u^2 - \psi_{uu} ).$$
Using \eqref{psiuv}, we compute
\begin{equation}\label{consf}
 \partial_v ( \frac{1}{2} \psi_u^2 - \psi_{uu} ) = -\frac{1}{4p} F_p(\psi) \psi_v
\end{equation}
where $F_p(s) := s (1 + \frac{1}{p} s)^{p-1}$.  From \eqref{psiup} we see that $F_p(\psi) = O_K(1)$, and thus by \eqref{psipsi} we have
$$ \partial_v ( \frac{1}{2} \psi_u^2 - \psi_{uu} ) = O_K(C_K^{1/2} / p).$$
From this, \eqref{psi0}, \eqref{psipsi} and the fundamental theorem of calculus we see that
\begin{equation}\label{psia}
|\frac{1}{2} \psi_u^2 - \psi_{uu}|(t,x) \leq A
\end{equation}
for all $0 \leq t \leq p(t_1-t_0)$ and $|x| \leq 50 \tau p$, and some $A \sim_K C_K^{1/2}$.

From \eqref{psiuv} we know that $\psi_u$ is decreasing in the $v$ direction, so from \eqref{psi0} we also have the upper bound
$$ \psi_u(t,x) \leq O_K(1)$$
for all $0 \leq t \leq p(t_1-t_0)$ and $|x| \leq 50 \tau p$.  To get a lower bound, suppose that $\psi_u(t,x) \leq -A^{1/2}$ for some 
$0 \leq t \leq p(t_1-t_0)$ and $|x| \leq 40 \tau p$.  Then from \eqref{psia} we have $\psi_{uu}(t,x) \geq 0$.  If we move backwards in the $u$ direction, we thus see that $\psi_u$ decreases; continuing this (by the usual continuity argument) until we hit the initial surface $t=0$ and applying \eqref{psi0} we conclude that $\psi_u(t,x) \geq -O_K(1)$, a contradiction if $C_K$ is large enough.  We thus conclude that $\psi_u \geq -A^{1/2}$ for $0 \leq t \leq p(t_1-t_0)$ and $|x| \leq 40 \tau p$.  Combining this with the upper bound and with \eqref{psia} we contradict \eqref{psipsi0} if $C_K$ is large enough, and the claim follows.
\end{proof}

Combining Lemmas \ref{pond}, \ref{ponder} and the definition of $K(t)$ we conclude that
$$ K(t) \ll_{K_{t_0}} 1$$
whenever $|t-t_0| \leq \tau$ and $T_0 \geq t \geq t_0 \geq 0$ or $-T_0 \leq t \leq t_0 \leq 0$.  Since $K(0) \ll 1$, we thus conclude on iteration that
$$ K(t) \ll 1$$
for all $t \in [-T_0,T_0]$.  Thus we have
\begin{align}
 |\phi(t,x)| &\leq 1 + \frac{\log p}{p} + O\left( \frac{1}{p} \right)\label{phibound-0} \\
 |\phi_u(t,x)|, |\phi_v(t,x)| &\ll 1\label{phibound-1} \\
 |\phi_{uu}(t,x)|, |\phi_{vv}(t,x)| &\ll p\label{phibound-2}
\end{align}
for all $(t,x) \in \diamondsuit_{T_0}$.

Now we revisit the conservation laws \eqref{cons} which were implicitly touched upon in the proof of Lemma \ref{ponder}.

\begin{lemma}[Approximate pointwise conservation law]\label{apc}  There exists $\tau > 0$ (depending on $T_0$ and the initial data) such that the following claim holds: whenever $\langle u_0, v_0 \rangle \in \diamondsuit_{T_0}$ is such that $\phi\langle u_0, v_0 \rangle \geq -1/2$, then
\begin{equation}\label{phiu-cons}
(\frac{1}{2} \phi_u^2 - \frac{1}{p} \phi_{uu})\langle u_0, v_0+r \rangle = (\frac{1}{2} \phi_u^2 - \frac{1}{p} \phi_{uu})\langle u_0, v_0 \rangle + O\left(\frac{\log p}{p}\right)
\end{equation}
and
$$ (\frac{1}{2} \phi_v^2 - \frac{1}{p} \phi_{vv})\langle u_0+r, v_0 \rangle = (\frac{1}{2} \phi_v^2 - \frac{1}{p} \phi_{vv})\langle u_0, v_0 \rangle + O\left(\frac{\log p}{p}\right)$$
for all $-\tau \leq r \leq \tau$.

If instead $\phi\langle u_0, v_0 \rangle \leq +1/2$, then we have
$$
(\frac{1}{2} \phi_u^2 + \frac{1}{p} \phi_{uu})\langle u_0, v_0+r \rangle = (\frac{1}{2} \phi_u^2 + \frac{1}{p} \phi_{uu})\langle u_0,v_0\rangle + O\left(\frac{\log p}{p}\right)
$$
and
$$ (\frac{1}{2} \phi_v^2 + \frac{1}{p} \phi_{vv})\langle u_0+r, v_0 \rangle = (\frac{1}{2} \phi_v^2 + \frac{1}{p} \phi_{vv})\langle u_0,v_0\rangle + O\left(\frac{\log p}{p}\right)$$
for all $-\tau \leq r \leq \tau$.
\end{lemma}

\begin{proof}  Let $\tau > 0$ be sufficiently small to be chosen later.  By sign reversal symmetry \eqref{sign-reverse} we may assume that $\phi\langle u_0,v_0\rangle \geq -1/2$.   By spatial reflection symmetry \eqref{space-reverse} it suffices to prove \eqref{phiu-cons}.  

Suppose first that $-1/2 \leq \phi\langle u_0,v_0\rangle \leq 1/2$, then by Lemma \ref{hold} we have $|\phi\langle u,v\rangle| \leq 0.9$ (say) whenever $|u-u_0|, |v-v_0| \leq 100 \tau$.  Applying \eqref{phiuv} and the fundamental theorem of calculus, we see that
$$ \phi_u\langle u_0+r,v_0\rangle = \phi_u\langle u_0,v_0\rangle + O\left( \frac{\log p}{p} \right)$$
for all $-\tau \leq r \leq \tau$; similarly, if one differentiates \eqref{phiuv} in the $u$ direction and applies the bound $|\phi\langle u,v\rangle| \leq 0.9$ as well as \eqref{phibound-1}, we obtain
$$ \phi_{uu}\langle u_0+r,v_0\rangle = \phi_{uu}\langle u_0,v_0\rangle + O(\log p)$$
and the claim \eqref{phiu-cons} follows.

Henceforth we assume $\phi\langle u_0,v_0\rangle > 1/2$.
By Lemma \ref{hold} (or \eqref{phibound-1}) we see that $\phi\langle u,v\rangle$ is positive when $|u-u_0|, |v-v_0| \leq 100 \tau$, so by making the ansatz
$$ \phi\langle u,v \rangle = p^{1/(p-1)} (1 + \frac{1}{p} \psi\langle p(u-u_0), p(v-v_0)\rangle)$$
as before, we see that $\psi$ obeys \eqref{psiuv} for $|u|, |v| \leq 100 \tau p$.  Also, from \eqref{phibound-0}, \eqref{phibound-1}, \eqref{phibound-2} (and the positivity of $\phi$) we see that
\begin{equation}\label{george}
|\psi_u\langle u,v\rangle|, |\psi_v\langle u,v\rangle|, |\psi_{uu}\langle u,v\rangle|, |\psi_{vv}\langle u,v\rangle| \ll 1
\end{equation}
and
\begin{equation}\label{george-2}
-p \leq  \psi\langle u,v\rangle \leq O(1)
\end{equation}
for all $|u|, |v| \leq 100 \tau p$.  Our objective is to show that
$$ (\frac{1}{2} \psi_u^2 - \psi_{uu})\langle 0,r\rangle = (\frac{1}{2} \psi_u^2 - \psi_{uu})\langle 0, 0 \rangle + O\left(\frac{\log p}{p}\right)$$
for all $-\tau p \leq r \leq \tau p$.  By \eqref{consf} and the fundamental theorem of calculus it suffices to show that
\begin{equation}\label{tawp}
\int_{-\tau p}^{\tau p} |F_p(\psi\langle 0,r\rangle)| |\psi_v\langle 0,r\rangle|\ dr \ll_T \frac{\log p}{p}.
\end{equation}
Applying \eqref{george} we can discard the $|\psi_v\langle 0,r\rangle|$ factor.  Meanwhile, from \eqref{psiexp}, \eqref{george}, and the fundamental theorem of calculus we have
$$ \int_{-\tau p}^{\tau p} e^{\psi\langle 0,r\rangle}\ dr \ll_T 1.$$
Observe that $F_p(x) \ll_T (\log p) e^x$ whenever $-100 \log p \leq x \leq O_T(1)$, and that $F_p(x) \ll p^{-50}$ when $x \leq -100 \log p$, and the claim \eqref{tawp} follows.
\end{proof}

We now use this law to show a more precise bound on $\phi_u$ and $\phi_v$ than is provided by \eqref{phibound-1}.  We first handle the case when $\phi$ has large derivative.

\begin{lemma}[Piecewise convergence, nondegenerate case]\label{piece}  Let $\eps > 0$, and let $I\subset [-T_0,T_0]$ be an interval such that $|\phi^{(\lin)}_u\langle u, v \rangle| \geq \eps$ for all $u \in I$ (note that $v$ is irrelevant here).  Then for each $v \in [-T_0,T_0]$, we have
\begin{equation}\label{phiu-cons2}
|\phi_u\langle u, v \rangle|^2 = |\phi^{(\lin)}_u\langle u, v \rangle|^2 + O\left( \frac{\log p}{p} \right)
\end{equation}
and
\begin{equation}\label{phiu-cons3}
\phi_{uu}\langle u, v \rangle = O( \log p )
\end{equation}
for all $u$ in $I$, excluding at most $O(1)$ intervals in $I$ of length $O_{\eps}( \frac{\log p}{p} )$.

Similarly with the roles of $u$ and $v$ reversed.
\end{lemma}

\begin{proof} Let $\tau > 0$ be a small number (depending on the initial data and $T_0$) to be chosen later.  We may assume that $p$ is sufficiently large depending on $\eps$, since the claim is trivial otherwise.  By space reflection symmetry \eqref{space-reverse} it will suffice to prove \eqref{phiu-cons2}, \eqref{phiu-cons3}; by time reversal symmetry \eqref{time-reverse} we may assume that $t = \frac{u+v}{2}$ is non-negative.

We introduce an auxiliary parameter $0 \leq T \leq T_0$, and only prove the claim for $t = \frac{u+v}{2}$ between $0$ and $T$; setting $T=T_0$ will then yield the claim.  We establish this claim by induction on $T$, incrementing $T$ by steps of $\tau$ at a time.  

Let us first handle the base case when $0 \leq T \leq 2\tau$.  Fix $v$.  By sign reversal symmetry \eqref{sign-reverse} and Lemma \ref{hold} we may assume that $\phi\langle u,v \rangle \geq -1/2$ whenever $t = \frac{u+v}{2}$ has magnitude at most $100\tau$.  Applying Lemma \ref{apc}, we conclude that
$$
(\frac{1}{2} \phi_u^2 - \frac{1}{p} \phi_{uu})\langle u, v \rangle = (\frac{1}{2} \phi_u^2 - \frac{1}{p} \phi_{uu})\langle u, -u \rangle + O\left(\frac{\log p}{p}\right)
$$
for all $-v-10\tau \leq u \leq -v+10\tau$.  From the hypotheses (a), (b) and \eqref{nlw} we see that $\phi_{uu}\langle u, -u \rangle = O(1)$, and so we have
\begin{equation}\label{soda}
 \phi_{uu}\langle u, v\rangle = \frac{p}{2} \left( \phi_u^2\langle u,v\rangle - \phi_u^2\langle u,-u\rangle + O\left( \frac{\log p}{p} \right)\right).
\end{equation}
If we write $f(t) := \phi_u\langle -v+2t,v\rangle$, $g(t) := \phi_u\langle -v+2t,v \rangle$ we thus have
\begin{equation}\label{fft}
 f'(t) = p\left( f(t)^2 - g(t)^2 + O\left( \frac{\log p}{p} \right) \right)
\end{equation}
for all $-5\tau \leq t \leq 5\tau$.  

Let $J := \{ 0 \leq t \leq T: -v+2t \in I\}$, thus $J$ is a (possibly empty) interval such that $\eps \leq |g(t)| \ll 1$ for all $t \in J$.  Also observe from the smoothness of the initial data and \eqref{nlw} that $g'(t) = O(1)$ for all $t \in J$.  Also from \eqref{phibound-1} we have $f(t) = O(1)$ for all $t \in J$.

Suppose first that $f(t_0) > 0$ and $f(t_0)^2 \geq g(t_0)^2 + C \frac{\log p}{p}$ for some $t_0 \in J$ and some sufficiently large $C$.  Then from the bounds on $g$ and \eqref{fft}, we have
$$ \partial_t (f(t_0)^2-g(t_0)^2) \gg_\eps p (f(t_0)^2 - g(t_0)^2).$$
A continuity argument (using Gronwall's inequality) then shows that $f(t)^2 - g(t)^2$ increases exponentially fast (with rate $\gg_\eps p$) as $t$ increases.  Since $f(t)^2 - g(t)^2$ is $O(1)$ and was $\gg \frac{\log p}{p}$ at $t_0$, we arrive at a contradiction unless $t_0$ lies within $O_{\eps}( \frac{\log p}{p})$ of the boundary of $J$.  Similarly if $f(t_0) < 0$ and $f(t_0)^2 \geq g(t_0)^2 + C \frac{\log p}{p}$  for some $t_0 \in J$.  We conclude that
$$ f(t)^2 \leq g(t)^2 + O\left( \frac{\log p}{p} \right)$$
for all $t \in J$ except for those $t$ which are within $O_{\eps}( \frac{\log p}{p} )$ of the boundary of $J$.

Now suppose that $f(t)^2 \leq \eps^2/2$, then we see from \eqref{fft} and the bounds on $f, g$ that $-f'(t) \gg_{T,\eps} p$; thus the set of $t \in J$ for which this occurs must be contained in a single interval of length $O_{\eps}(\frac{1}{p})$.  

Next, if $\eps^2/2 \leq f(t)^2 \leq g(t)^2 - C \frac{\log p}{p}$, then from \eqref{fft} we obtain a bound of the form
$$ \pm \partial_t (g(t)^2-f(t)^2) \gg_\eps p (g(t)^2 - f(t)^2),$$
where $\pm$ is the sign of $f(t)$.  Applying the continuity and Gronwall argument again, either forwards or backwards in time as appropriate, we see that this event can only occur either within $O_{\eps}(\frac{\log p}{p})$ of the boundary of $J$, or on an interval of length $O_{\eps}(\frac{\log p}{p})$ adjacent to the interval where $f(t)^2 \leq \eps^2/2$.

Putting all of this together, we see that $f(t)^2 = g(t)^2 + O\left( \frac{\log p}{p} \right)$ for all $t \in J$ outside of at most $O(1)$ intervals of length $O_{\eps}(\frac{\log p}{p})$.  This gives the desired bound \eqref{phiu-cons2}.  The bound \eqref{phiu-cons3} then follows from \eqref{soda}.

Now suppose inductively that $T > 2\tau$, and that the claim has already been shown for $T-\tau$.  By inductive hypothesis we only need to establish the claim for $t \in [T-\tau,T]$.  Fix $v$.  By sign reversal symmetry \eqref{sign-reverse} and Lemma \ref{hold} we may assume that $\phi\langle u,v \rangle \geq -1/2$ whenever $t = \frac{u+v}{2}$ lies within $100\tau$ of $T$.

We can now repeat the previous arguments, except that the interval $J$ must first be subdivided by removing the $O(1)$ subintervals of $J$ of length $O_{\eps}(\frac{\log p}{p})$ for which \eqref{phiu-cons2}, \eqref{phiu-cons3} (with $v$ replaced by $v-\tau$) already failed (and which are provided by the inductive hypothesis), and then working on each remaining subinterval of $J$ separately.  Note on each such interval we still have the ODE \eqref{fft} (using the inductive hypothesis \eqref{phiu-cons2}, \eqref{phiu-cons3} as a substitute for control of the initial data).  We omit the details. 
\end{proof}

Now we handle the opposing case when $\phi$ has small derivative.

\begin{lemma}[Piecewise convergence, degenerate case]\label{piece-2}  Let $\eps > 0$, and let $I\subset [-T_0,T_0]$ be an interval such that $|\phi^{(\lin)}_u\langle u, v \rangle| \leq \eps$ for all $u \in I$ (again, $v$ is irrelevant).  Then one has
\begin{equation}\label{phiu-cons2a}
|\phi_u\langle u,v\rangle| \ll \eps 
\end{equation}
and
\begin{equation}\label{phiu-cons3b}
|\phi_{uu}\langle u,v \rangle| \ll \eps^2 p
\end{equation}
whenever $\langle u,v \rangle \in \diamondsuit_{T_0}$ and $u \in I$.  Similarly with the roles of $u$ and $v$ reversed.
\end{lemma}

\begin{proof} This is very similar to Lemma \ref{piece}, in that we first establish the base case $0 \leq T \leq 4\tau$ and then induct by steps of $\tau$, where $\tau > 0$ is a fixed timestep independent of $T$ and $p$.  Whereas in the proof of Lemma \ref{piece} we did the base case in detail and left the inductive step to the reader, here we shall leave the base case to the reader and do the inductive step in detail.  Thus, assume $T > 4\tau$ and that the claim has already been proven for $T-\tau$ and $T-2\tau$.  We may assume $\eps < \tau$ since the claim follows from \eqref{phibound-1}, \eqref{phibound-2} otherwise.

By inductive hypothesis and time and space reversal symmetry \eqref{time-reverse}, \eqref{space-reverse} we only need to establish the claims \eqref{phiu-cons2a}, \eqref{phiu-cons3b} for $t = \frac{u+v}{2} \in [T-\tau,T]$.  By the sign reversal symmetry \eqref{sign-reverse} and Lemma \ref{hold} we may assume that $\phi\langle u,v \rangle > -1/2$ whenever $t \in [T-100\tau,T+100\tau]$.  We can assume that $p$ is large compared to $T_0$, $\eps$, and the initial data since the claim is vacuous otherwise.

Let $J := \{ u \in I: t = \frac{u+v}{2} \in [T-\tau,T] \}$.
Observe (from the smoothness of the initial data) that $|\phi_u\langle u, u \rangle| \ll \eps$ whenever $u$ lies within $\eps$ of $J$.  By inductive hypothesis (replacing $\eps$ by $O(\eps)$), we conclude that
$$ |\phi_u\langle u, v-2\tau \rangle )| \ll \eps $$
and
$$ |\phi_{uu}\langle u, v-2\tau\rangle )| \ll \eps^2 p $$
for all in the $\eps$-neighbourhood of $J$.  Applying Lemma \ref{apc}, we conclude that
\begin{equation}\label{jump}
(\frac{1}{2} \phi_u^2 - \frac{1}{p} \phi_{uu})\langle u, v\rangle = O( \eps^2 )
\end{equation}
for all $t$ in the $\eps$-neighbourhood of $J$.  Thus, if $f(u) := \phi_u\langle u, v \rangle$, then we have
$$ f'(u) = p ( f^2(u) + O(\eps^2) )$$
for all $u$ in the $\eps$-neighbourhood of $J$.

The ODE $f'(u) = \frac{p}{2} f^2(u)$ blows up either forward or backward in time within a duration of $O_\eps(1/p)$ as soon as $|f(u)|$ exceeds $\eps$.  From this and a continuity and comparison argument, we see that $|f(u)|$ cannot exceed $C \eps$ for $u \in J$ for some constant $C$ depending only on $u$, thus $f(u) = O(\eps)$ for all $u \in J$.  Applying \eqref{jump} we close the induction as required; the base case is similar.
\end{proof}

\begin{remark} The \emph{a priori} estimates here did not use the full force of the hypotheses (a)-(c); the condition (c) was not used at all, and the strict barrier condition (b) could be replaced by the non-strict condition $|\phi_0(x)| \leq 1$.  Also, a careful examination of the dependence of the implied constants on the initial data, combined with a standard limiting argument using the usual local-wellposedness theory reveals that (a) can be replaced with a $C^2 \times C^1$ condition on the initial data $(\phi_0,\phi_1)$.  However, we use the hypotheses (a)-(c) more fully in the uniqueness theory of the previous section, and the compactness arguments in the next section.
\end{remark}

\section{Compactness}\label{compact-sec}

Now we can prove Theorem \ref{main-thm}. Fix $T_0 > 0$.  From Lemma \ref{hold}, the solutions $\phi = \phi^{(p)}$ are equicontinuous and uniformly bounded on the region $\diamondsuit_{T_0} := \{ (t,x): |t|+|x| \leq T_0 \}$, and hence (by the Arzel\'a-Ascoli theorem) precompact in the uniform topology on this region.  In view of the uniqueness theory in Section \ref{unique-sec}, we see that to show Theorem \ref{main-thm}, it suffices to show that any limit point of this sequence obeys the properties (i)-(vi) on $\diamondsuit_{T_0}$.  Accordingly, let $p_n \to \infty$ be a sequence such that $\phi^{(p_n)}$ converges uniformly to a limit $\phi$.  

We can now quickly verify several of the required properties (i)-(vi).
From \eqref{phibound-0} we obtain the property barrier condition (iv), while from Lemma \ref{hold} we have the Lipschitz condition (i).  From the strict barrier hypothesis (b) and Lemma \ref{hold}, we know that the $|\phi^{(p_n)}|$ stay bounded away from $1$ in a neighbourhood of the initial interval $\{ (0,x): -T_0 \leq x \leq T_0 \}$, and so the nonlinearity $|\phi^{(p_n)}|^{p_n-1} \phi^{(p_n)}$ converges uniformly to zero in this neighbourhood.  Because of this and \eqref{nlw}, $\phi^{(p_n)}$ converges uniformly to $\phi^{(\lin)}$ in this neighbourhood, yielding the initial condition (iii).

\subsection{The defect measure condition}

Now we verify (v).  Suppose $(t_0,x_0)$ is a point in $\diamondsuit$ such that $|\phi(t_0,x_0)| < 1$.  Then by Lemma \ref{hold} and uniform convergence, we can find a neighbourhood $B$ of $(t_0,x_0)$ in $\diamondsuit$ and a constant $c < 1$ such that $|\phi^{(p)}(t,x)| \leq c$ for all $(t,x) \in B$ and all sufficiently large $p$.  In particular, the nonlinearity in \eqref{nlw} converges uniformly to zero on $B$ as $p \to \infty$.  Taking limits, we see that $-\phi_{tt} + \phi_{xx} = 0$ on $B$ in the sense of distributions.  Taking unions over all such $B$, and using null coordinates we conclude that
the distribution $-\phi_{tt}+\phi_{xx}$ is supported on the set $\{ (t,x) \in \diamondsuit: |\phi(t,x)| = 1 \}$.

Next, we consider the neighbourhood of a point $(t_0,x_0)$ where $\phi(t_0,x_0)=+1$ (say).  Then by Lemma \ref{hold} and uniform convergence, we can find a diamond $D$ centred at $(t_0,x_0)$ (with length bounded below uniformly in $(t_0,x_0)$) such that $\phi^{(p)}(t,x)$ is non-negative for all $(t,x) \in D$ and all sufficiently large $p$.  In particular, the nonlinearity in \eqref{phiuv} is non-negative on this diamond, which implies by the fundamental theorem of calculus that
$$ \phi^{(p)}( t,x ) - \phi^{(p)}(t-r,x-r) - \phi^{(p)}(t-s,x+s) + \phi^{(p)}(t-r-s,x-r+s) \geq 0$$
whenever $(t,x), (t-r,x-r), (t-s,x+s), (t-r-s,x-r+s)$ lie in $D$.  Taking uniform limits, we conclude that the same statement is true for $\phi$.  By the usual Lebesgue-Stieltjes measure construction (adapted to two dimensions) we thus see that
\begin{align*}
\phi( t,x ) - \phi(t-r,x-r) - &\phi(t-s,x+s) + \phi(t-r-s,x-r+s) \\
&= \mu_+( \{ (t-r'-s',x-r'+s': 0 \leq r' \leq r; 0 \leq s' \leq s \} )
\end{align*}
for some positive finite measure $\mu_+$ on $D$, which implies that $-\phi_{tt} + \phi_{xx} = \mu_+$ in the sense of distributions on $D$.  Similarly when $\phi(t_0,x_0)=-1$ (now replacing $\mu_+$ by $-\mu_-$).  Piecing together these diamonds $D$ and neighbourhoods $B$ we obtain the claim.

\subsection{The reflection condition}

Now we verify (vi).  By space reflection symmetry \eqref{space-reverse} it suffices to show \eqref{phiu}. 

Let us first consider the region where $\phi_u^{(\lin)}\langle u, v\rangle$ vanishes.  Applying Lemma \ref{piece-2} and taking weak limits, we see that $\phi_u\langle u,v \rangle$ vanishes almost everywhere when $\phi_u^{(\lin)}\langle u, v\rangle$ vanishes, which of course gives \eqref{phiu} in this region.    As the countable union of null sets is still null, it thus suffice to verify \eqref{phiu} for almost every $(t,x)$ in the parallelogram $P := \{ \langle u,v\rangle \in \diamondsuit_{T_0}: u \in I \}$, whenever $I$ is an interval such that 
$$|\phi_u^{(\lin)}\langle u,v\rangle| \geq \eps$$ 
for all $u \in I$ ($v$ is irrelevant) and some $\eps > 0$.  Applying Lemma \ref{piece} and taking square roots, we know that for all $p$, we have
\begin{equation}\label{phiu-osc}
|\phi^{(p)}_u \langle u,v \rangle| =  |\phi^{(\lin)}_u \langle u,v \rangle| + O_T( \frac{\log^{1/2} p}{p^{1/2}} )
\end{equation}
for all $\langle u, v \rangle \in \diamondsuit_{T_0}$ with $u \in I$, where we exclude for each fixed choice of $v$, a union $I_v \subset I$ of $O(1)$ intervals of length $O_{\eps}( \frac{\log p}{p} )$ from $I$.  

We would like to take limits as $p = p_n \to \infty$, but we encounter a technical difficulty: while we know that $\phi^{(p)}_u$ converges weakly to $\phi_u$, this does not imply that $|\phi^{(p)}_u|$ converges weakly to $|\phi_u|$, due to the possibility of increasing oscillation of sign\footnote{If one does not address this oscillation issue, one can only get the lower bound in \eqref{phiu} rather than equality.} in $\phi^{(p)}_u$.  The fact that \eqref{phiu-osc} only fails on a bounded number of short intervals for each $v$ rules out oscillation in the $u$ direction, but one must also address the issue of oscillation in the $v$ direction.  Fortunately, from \eqref{phiuv} we have some monotonicity of $\phi^{(p)}_u$ in $v$ that allows us to control this possibility.

We turn to the details.  As $\phi$ is Lipschitz, we can cover the parallelogram $P$ by a bounded number of open diamonds $D$ in $P$, on which each $\phi$ varies by at most $0.1$ (say).  If $\phi$ takes any value between $-1/2$ and $1/2$ on a diamond $D$, then by \eqref{phiuv} $\phi$ solves the free wave equation on $D$, so in particular $\phi_u$ is constant in $v$ (and agrees with $\frac{1}{2} (\phi_1 + \partial_x \phi_0)$ whenever the diamond intersects the initial surface $\{t=0\}$).  Thus it suffices to establish the claim on those diamonds $D$ on which $\phi$ avoids the interval $[-1/2,1/2]$; by the symmetry $\phi \to -\phi$ we may assume that $\phi \geq 1/2$ on $D$, and hence (for $n$ large enough) $\phi^{(p_n)}$ is also positive.  By \eqref{phiuv}, we conclude that $\phi^{(p_n)}_u$ is decreasing in the $v$ direction.

Let $\delta > 0$ be a small number.  We can partition the diamond $D$ into $O_T(\delta^{-2})$ subdiamonds of length $\delta$ in a regular grid pattern.  Fix $n$ sufficiently large depending on $\delta,\eps$, and call a subdiamond \emph{totally positive} (with respect to $n$) if $\phi^{(p_n)}_u > 0$ at every point on this subdiamond; similarly define the notion of a subdiamond being \emph{totally negative}.  Call a subdiamond \emph{degenerate} if it is neither totally positive nor totally negative (i.e. it attains a zero somewhere in the diamond).  We claim that at most $O_{\eps}(\delta^{-1})$ degenerate subdiamonds.  To see this, let $d$ be a degenerate sub-diamond.  Since $\phi^{(p_n)}_u$ is decreasing in the $v$ direction, we know that $\phi^{(p_n)}_u$ must be negative in at least one point on the northwest edge of $d$, and positive in at least one point on the southeast edge.  Suppose that $\phi^{(p_n)}_u$ is negative at every point on the northwest edge.  Then from the monotonicity of $\phi^{(p_n)}_u$ in the $v$ direction, we see that there can be at most one degenerate subdiamond of this type on each northwest-southeast column of subdiamonds; thus there are only $O(\delta^{-1})$ subdiamonds of this type.  Thus we may assume that $\phi^{(p_n)}_u$ changes sign on the northwest edge.  But on the line $\ell$ that that edge lies on, we have \eqref{phiu-osc} holding in $D$ outside of $O(1)$ intervals of length $O_{\eps}( \frac{\log p_n}{p_n} )$; also, by hypothesis, we have 
$$ |\phi^{(\lin)}_u\langle u,v\rangle| \geq \eps$$ 
for $\langle u,v\rangle$ in $D$.  We conclude (for $n$ large enough) that there are at most $O(1)$ subdiamonds with northwest edge lying on this line $\ell$ for which $\phi^{(p_n)}_u$ changes sign on this edge.  Summing over all $O(\delta^{-1})$ possible edges, we obtain the claim.

Fix $\delta$, and let $n \to \infty$.  The set of subdiamonds on which $\phi^{(p_n)}$ is totally positive or totally negative can change with $n$; however there are only a finite number of possible values for this set for fixed $\delta$.  Hence, by the infinite pigeonhole principle, we may refine the sequence $p_n$ and assume that these sets are in fact independent of $n$.  For any totally positive or totally negative diamond, $\phi^{(p_n)}_u$ has a definite sign; since $\phi^{(p_n)}_u$ converges weakly to $\phi_u$, we conclude that $|\phi^{(p_n)}_u|$ converges weakly to $|\phi_u|$.  Since there are no sign changes on this diamond, \eqref{phiu-osc} must hold throughout the subdiamond (by the intermediate value theorem); we thus conclude that \eqref{phiu} holds on any such subdiamond.  Since the measure of all the degenerate sub-diamonds is $O_{\eps}(\delta)$, we conclude that \eqref{phiu} holds on $D$ outside of a set of measure $O_{\eps}(\delta)$.  Letting $\delta \to 0$ we obtain the claim.

\subsection{Piecewise smoothness}

The only remaining property we need to verify is (ii).  By spatial reflection symmetry \eqref{space-reverse}, it suffices to show that for each $v \in [-T_0,T_0]$, the map $u \mapsto \phi \langle u, v \rangle$ is piecewise smooth on $[-T_0,T_0]$, with only finitely many pieces.

From (c), we know that $\phi^{(\lin)}_u \langle u, v \rangle$ vanishes for $u$ in a finite union of intervals and points in $[-T_0,T_0]$.  On any one of these intervals, we know from (vi) that $\phi_u \langle u,v\rangle$ also vanishes almost everywhere, which by the Lipschitz nature of $\phi$ and the fundamental theorem of calculus ensures that $\phi \langle u,v \rangle$ is constant in $u$ on each of these intervals, for any fixed $v$.  So it will suffice to verify the piecewise smoothness of $u \mapsto \phi \langle u,v \rangle$ for any $v$ and on any compact interval $I$ of $u$ for which $\phi^{(\lin)}_u \langle u, v \rangle$ is bounded away from zero, so long as the number of pieces is bounded uniformly in $I$ and $v$.

Fix $I$ and $v$.  By hypothesis, we can find $\eps > 0$ such that $|\phi^{(\lin)}_u \langle u, v \rangle| \geq \eps$ for all $u \in I$; in particular, $\phi^{(\lin)}_u$ does not change sign on this interval.  By Lemma \ref{piece}, we conclude for each $n$ that $|\phi^{(p_n)}_u\langle u, v \rangle| = |\phi^{(\lin)}_u\langle u, v \rangle| + O\left( \frac{\log^{1/2} p_n}{p_n^{1/2}} \right)$  for $u \in I$ outside of $O(1)$ intervals of length $O_\eps( \frac{\log p_n}{p_n} )$ intersecting $I$.
  
By pigeonholing, we may assume that the number $k = O(1)$ of such intervals is constant; denoting the midpoints of these intervals by $u^{(p_n)}_1 < \ldots < u^{(p_n)}_k$; without loss of generality we may take $u^{(p_n)}_1$ and $u^{(p_n)}_k$ to be the endpoints of $I$.  We may assume from the Bolzano-Weierstrass theorem and passing to a further subsequence that each of the $u^{(p_n)}_j$ converge to some limit $u_j$.  

Between $u^{(p_n)}_j$ and $u^{(p_n)}_{j+1}$, excluding those $u$ lying within $O_\eps( \frac{\log p_n}{p_n} )$ of either endpoint, we may write $\phi^{(p_n)}_u\langle u, v \rangle = \epsilon^{(p_n)}_j \phi^{(\lin)}_u\langle u, v \rangle + O\left( \frac{\log^{1/2} p_n}{p_n^{1/2}} \right)$, where $\epsilon^{(p_n)}_j \in \{-1,+1\}$.  By a further pigeonholing we may take $\epsilon^{(p_n)}_j = \epsilon_j$ independent of $n$.  Using the fundamental theorem of calculus and then taking limits, we conclude that $\phi\langle u,v\rangle$ is piecewise smooth for $u \in I$, with possible discontinuities at $u_1,\ldots,u_k$, and with $\phi\langle u,v\rangle$ equal to $\epsilon_j \phi^{(\lin)}_u\langle u,v\rangle$ on the interval $(u_j,u_{j+1})$ for any $1 \leq j < k$.  The claim follows, and the proof of Theorem \ref{main-thm} is complete.

\end{document}